\documentclass[a4paper,11pt]{article}
\AtBeginDvi{}

\usepackage[dvipdfmx]{graphics}
\usepackage{amsmath,amssymb,amsthm}

\theoremstyle{plain}
\newtheorem{theorem}{Theorem}
\newtheorem{proposition}{Proposition}[section]
\newtheorem{lemma}[proposition]{Lemma}
\theoremstyle{definition}
\newtheorem{remark}[proposition]{Remark}

\def\TA{\mathbin{\ast}}
\def\TP#1{{\vphantom{#1}}^{\mathit{t}}{#1}}
\def\TW#1{t_{#1}}

\title{On the self-intersection number of the nonsingular models of rational cuspidal plane curves}
\author{Keita Tono}
\date{}

\begin{document}
\maketitle
\def\thefootnote{}
\footnotetext{\textit{2010 Mathematics Subject Classification.} 14H50}
\footnotetext{\textit{Key words and phrases.} plane curve, cusp, self-intersection number.}
\footnotetext{Supported by JSPS Grant-in-Aid for Scientific Research (22540040) and the Fuujukai Foundation.}
\def\thefootnote{1}
\begin{abstract}
In this paper,
we consider rational cuspidal plane curves
having at least three cusps.
We give an upper bound of
the self-intersection number of
the proper transforms of such curves
via the minimal embedded resolution of the cusps.
For a curve having exactly three cusps,
we show that the self-intersection number is equal to the bound
if and only if
the curve coincides with the quartic curve having three cusps.
\end{abstract}
%%%%%%%%%%%%%%%%%%%%%%%%%%%%%%%%%%%%%%%%%%%%%%%%%%%%%%%%%%%%%%%%%%%%%%%%%%%%%%%
\section{Introduction}
%%%%%%%%%%%%%%%%%%%%%%%%%%%%%%%%%%%%%%%%%%%%%%%%%%%%%%%%%%%%%%%%%%%%%%%%%%%%%%%
Let $C$ be an algebraic curve on $\mathbf{P}^2=\mathbf{P}^2(\mathbf{C})$.
A singular point of $C$ is said to be a \emph{cusp}
if it is a locally irreducible singular point.
We say that $C$ is \emph{cuspidal}
if $C$ has only cusps as its singular points.
Suppose that $C$ is rational and cuspidal.
Let $C'$ be the proper transform of $C$
via the minimal embedded resolution of the cusps.
Let $(C')^2$ denote its self-intersection number.
For instance,
$(C')^2=d$ if $C$ is the rational cuspidal plane curve
defined by the equation
$x^d=y^{d-1}z$,
where $d>2$ and $(x,y,z)$ are homogeneous coordinates of $\mathbf{P}^2$.
We estimate $(C')^2$ in the following way.

\begin{theorem}\label{thm1}
Let $C$ be a rational cuspidal plane curve with $n$ cusps.
If $n\ge 3$, then $(C')^2\le 7-3n$.
\end{theorem}

\begin{remark}\label{rmk}
It was proved in \cite[Theorem 3.5]{hartshorne:si} that
if $\Gamma$ is a smooth curve of genus $g\ge1$
on a smooth rational projective surface $S$ ($S\ne\mathrm{P}^2$ if $g=1$),
then $\Gamma^2\le4g+4$.
From this fact we infer that
if $C$ is a cuspidal plane curve of genus $g\ge1$ with $n$ cusps
then $(C')^2\le 4g+4-2n$.
It was shown in \cite{sst} that
for given integers $g,n$ with $g\ge0$, $1\le n\le2g+2$
there exist
sequences of cuspidal plane curves $C$ of genus $g$ with $n$ cusps
such that $(C')^2=4g+4-2n$.
\end{remark}

Let $C$ be a rational cuspidal plane curve with $n$ cusps.
It was proved in \cite{tono:nc} that $n\le8$.
We denote by $\overline{\kappa}=\overline{\kappa}(\mathbf{P}^2\setminus C)$
the logarithmic Kodaira dimension of the complement of $C$.
By \cite{wakabayashi}, we see
$\overline{\kappa}=2$ if $n\ge 3$.
Moreover, if $n=2$, then $\overline{\kappa}\ge 0$.
If $n=1$, then it was proved in \cite{yoshihara} that
$\overline{\kappa}=-\infty$ if and only if $(C')^2>-2$.
If $n=2$, then $(C')^2\le 0$ by \cite{tono:bck2}.
From these facts and Theorem~\ref{thm1},
$(C')^2$ is bounded from above if $\overline{\kappa}\ne-\infty$.

There are no known examples of
rational cuspidal plane curves
having more than $4$ cusps.
There is only one known
rational cuspidal plane curve  $C$ with $4$ cusps.
The curve $C$ is a quintic curve with $(C')^2=-7$
(\cite[Theorem 2.3.10]{namba}).
In \cite{fz:dd2} (resp.\ \cite{fz:dd3}, \cite{fenske}),
a sequence of rational cuspidal plane curves $C$ of degree $d$
with three cusps was constructed,
where $d\ge4$ (resp.\ $d=2k+3$, $d=3k+4$, $k\ge1$).
They satisfy $(C')^2=2-d$
(resp.\ $(C')^2=-k-2$, $(C')^2=-k-3$).
The bound given by Theorem~\ref{thm1}
is the best possible one for the case in which $n=3$
as the quartic curve $C$ with three cusps satisfies $(C')^2=-2$.
Moreover, we prove the following:

\begin{theorem}\label{thm2}
Let $C$ be a rational cuspidal plane curve with three cusps.
Then $(C')^2=-2$ if and only if
$C$ coincides with the quartic curve having three cusps.
\end{theorem}

In \cite{tono:orevkov},
it was proved that
the rational cuspidal plane curves $C$ with
$n=1$, $\overline{\kappa}=2$ and $(C')^2=-2$
coincide with those constructed by Orevkov in \cite{orevkov}.
In \cite{tono:bck2},
the rational cuspidal plane curves $C$ with
$n=2$, $\overline{\kappa}=2$ and $(C')^2=-1$ were classified.
Theorem~\ref{thm2} is a similar result for the case in which $n=3$.
%%%%%%%%%%%%%%%%%%%%%%%%%%%%%%%%%%%%%%%%%%%%%%%%%%%%%%%%%%%%%%%%%%%%%%%%%%%%%%%
\section{Preliminaries}
%%%%%%%%%%%%%%%%%%%%%%%%%%%%%%%%%%%%%%%%%%%%%%%%%%%%%%%%%%%%%%%%%%%%%%%%%%%%%%%
In this section we prepare the proof of our theorems.
%%%%%%%%%%%%%%%%%%%%%%%%%%%%%%%%%%%%%%%%%%%%%%%%%%%%%%%%%%%%%%%%%%%%%%%%%%%%%%%
\subsection{Linear chains}\label{ssec:linear-chains}
%%%%%%%%%%%%%%%%%%%%%%%%%%%%%%%%%%%%%%%%%%%%%%%%%%%%%%%%%%%%%%%%%%%%%%%%%%%%%%%
Let $D$ be a divisor on a smooth projective surface $V$.
Let $D_1,\ldots,D_r$ be the irreducible components of $D$.
We call $D$ an \emph{SNC-divisor} if
$D$ is a reduced effective divisor,
each $D_i$ is smooth,
$D_iD_j\le 1$ for distinct $D_i,D_j$,
and $D_i\cap D_j\cap D_k=\varnothing$ for distinct $D_i,D_j,D_k$.
Assume that $D$ is an SNC-divisor.
We use the following notation and terminology
(cf.\ \cite[Section 3]{fujita} and \cite[Chapter 1]{mits}).
A blow-up at a point $P\in D$
is said to be \emph{sprouting} (resp.\ \emph{subdivisional})
\emph{with respect to} $D$
if $P$ is a smooth point (resp.\ node) of $D$.
We also use this terminology
for the case in which $D$ is a point.
By definition, the blow-up is subdivisional in this case.

Assume that each $D_i$ is rational.
Let $\Gamma=\Gamma(D)$ denote the dual graph of $D$.
We give the vertex corresponding to a component $D_i$
the weight $D_i^2$.
We sometimes do not distinguish between $D$
and its weighted dual graph $\Gamma$.
Assume that $\Gamma$ is connected and linear.
In case where $r>1$,
the weighted linear graph $\Gamma$ together with
a direction from an endpoint to the other
is called a \emph{linear chain}.
By definition,
the empty graph $\varnothing$
and a weighted graph consisting of a single vertex without edges
are linear chains.
If necessary, renumber $D_1,\ldots,D_r$ 
so that the direction of the linear chain $\Gamma$ is from $D_1$ to $D_r$
and $D_iD_{i+1}=1$ for $i=1,\ldots,r-1$.
We denote $\Gamma$ by $[-D_1^2,\ldots,-D_r^2]$.
We sometimes write $\Gamma$ as $[D_1,\ldots,D_r]$.
The linear chain $\Gamma$ is called \emph{admissible} 
if $\Gamma\ne\varnothing$ and $D_i^2\le-2$ for each $i$.
We define
the \emph{discriminant} $d(\Gamma)$ of $\Gamma$
as the determinant of the $r\times r$ matrix $(-D_i D_j)$.
We set $d(\varnothing)=1$.

Let $A=[a_1,\ldots,a_r]$ be a linear chain.
We use the following notation if $A\ne\varnothing$:
\[
\TP{A}:=[a_r,\ldots,a_1],\ 
\overline{A}:=[a_2,\ldots,a_r],\ 
\underline{A}:=[a_1,\ldots,a_{r-1}].
\]
The discriminant $d(A)$ has the following properties (\cite[Lemma 3.6]{fujita}).
\begin{lemma}\label{lem:det1}
Let $A=[a_1,\ldots,a_r]$ be a linear chain,
where $a_1,\ldots,a_r$ are integers.
\begin{enumerate}
\item
If $r>1$, then
$d(A)=a_1 d(\overline{A})-d(\overline{\overline{A}})=d(\TP{A})=a_r d(\underline{A})-d(\underline{\underline{A}})$.
\item
If $r>1$, then
$d(\overline{A})d(\underline{A})-d(A)d(\underline{\overline{A}})=1$.
\item
If $A$ is admissible,
then $\gcd(d(A),d(\overline{A}))=1$ and $d(A)>d(\overline{A})>0$.
\end{enumerate}
\end{lemma}
%
%%%%%%%%%%%%%%%%%%%%%%%%%%%%%%%%%%%%%%%%%%%%%%%%%%%%%%%%%%%%%%%%%%%%%%%%%%%%%%%

Let $A$ be an admissible linear chain.
The rational number $e(A):=d(\overline{A})/d(A)$
is called the \emph{inductance} of $A$.
By \cite[Corollary 3.8]{fujita}, the function
$e$ defines a one-to-one correspondence between the
set of all the admissible linear chains and the set of rational numbers
in the interval $(0,1)$.
For a given admissible linear chain $A$,
the admissible linear chain $A^{\ast}:=e^{-1}(1-e(\TP{A}))$ is called
the \emph{adjoint} of $A$ (\cite[3.9]{fujita}).
Admissible linear chains and their adjoints have the following properties
(\cite[Corollary 3.7, Proposition 4.7]{fujita}).
\begin{lemma}\label{lem:indf}
Let $A$ and $B$ be admissible linear chains.
\begin{enumerate}
\item
If $e(A)+e(B)=1$, then $d(A)=d(B)$ and $e(\TP{A})+e(\TP{B})=1$.
\item
We have $A^{\ast\ast}=A$, $\TP{(A^{\ast})}=(\TP{A})^{\ast}$ and
$d(A)=d(A^{\ast})=d(\overline{A^{\ast}})+d(\underline{A})$.
\end{enumerate}
\end{lemma}

For integers $m$, $n$ with $n\ge 0$, we define 
$[m_n]=[\overbrace{m,\ldots,m}^n]$, $\TW{n}=[2_n]$.
For non-empty linear chains $A=[a_1,\ldots,a_r]$, $B=[b_1,\ldots,b_s]$,
we write
$A\TA B=[\underline{A},a_r+b_1-1,\overline{B}]$,
$A^{\ast n}=\overbrace{A\TA\cdots\TA A}^n$,
where $n\ge 1$ and
$a_1,\ldots,a_r,b_1,\ldots,b_s$ are integers.
We remark that $(A\TA B)\TA C=A\TA(B\TA C)$
for non-empty linear chains $A$, $B$ and $C$.
By using Lemma~\ref{lem:det1} and Lemma~\ref{lem:indf},
we can show the following lemma.
\begin{lemma}\label{lem:adj}
Let $A=[a_1+1,\ldots,a_r+1]$ be an admissible linear chain,
where $a_1,\ldots,a_r$ are positive integers.
\begin{enumerate}
\item
For a positive integer $n$, we have $[A,n+1]^{\ast}=\TW{n}\TA A^{\ast}$.
\item
We have $A^{\ast}=\TW{a_r}\TA\cdots\TA\TW{a_1}$.
\end{enumerate}
\end{lemma}
%%%%%%%%%%%%%%%%%%%%%%%%%%%%%%%%%%%%%%%%%%%%%%%%%%%%%%%%%%%%%%%%%%%%%%%%%%%%%%%
\subsection{Vanishing theorem and Zariski decomposition}\label{ssec:zariski}
%%%%%%%%%%%%%%%%%%%%%%%%%%%%%%%%%%%%%%%%%%%%%%%%%%%%%%%%%%%%%%%%%%%%%%%%%%%%%%%
Let $V$ be a smooth projective surface,
$K$ a canonical divisor
and $D\ne0$ a $\mathbf{Q}$-divisor on $V$.
Write $D$ as $D=\sum_{i=1}^r q_iD_i$,
where $q_i\in\mathbf{Q}\setminus\{0\}$ and
all $D_i$'s are distinct irreducible curves.
The divisor $D$ is said to be \emph{numerically effective}
(\emph{nef}, for short)
if $DC\ge0$ for all curves $C$ on $V$.
We define
$\lfloor D\rfloor=\sum_i\lfloor q_i\rfloor D_i$
and
$\lceil D\rceil=-\lfloor-D\rfloor$,
where $\lfloor q_i\rfloor$ is the greatest integer less than or equal to $q_i$.
We will use the following vanishing theorem
(\cite{miyaoka}, \cite[Theorem 5.1]{sakai:weil}).
\begin{theorem}\label{thm:miyaoka}
Let $D$ be a nef $\mathbf{Q}$-divisor on a smooth projective surface V
and $K$ a canonical divisor on $V$.
If $D^2>0$, then $h^i(V,K+\lceil D\rceil)=0$ for $i>0$.
\end{theorem}

We denote by $\mathbf{Q}(D)$
the $\mathbf{Q}$-vector space generated by $D_1,\ldots,D_r$.
The divisor $D$ is said to be \emph{contractible}
if the intersection form defined on $\mathbf{Q}(D)$
is negative definite (\cite[Section 6]{fujita}).
The divisor $D$ is said to be \emph{pseudo-effective}
if $DH\geq0$ for all nef divisors $H$ on $V$.
By \cite[Theorem 6.3]{fujita},
if $D$ is pseudo-effective,
then there exists an effective $\mathbf{Q}$-divisor $N$
satisfying the following conditions.
\begin{enumerate}
\item
$N$ is contractible if $N\ne0$.
\item
$H=D-N$ is nef.
\item
$HE=0$ for all irreducible components $E$ of $N$.
\end{enumerate}
The divisor $N$ is determined by
the numerical equivalence class of $D$
by \cite[Lemma 6.4]{fujita}.
The decomposition $D=H+N$ is called the \emph{Zariski decomposition} of $D$.
The divisor $N$ (resp.\ $H$)
is called the \emph{negative part} (resp.\ \emph{nef part}) of $D$.

From now on,
we assume that $V$ and $D$ satisfy the following conditions.
\begin{enumerate}
\item[(Z1)]
$D\ne0$ is an SNC-divisor such that $K+D$ is pseudo-effective.
\item[(Z2)]
Each ($-1$)-curve $E\le D$ satisfies $(D-E)E> 2$ or,
$(D-E)E=2$ and $E$ intersects only a single irreducible component of $D$.
\end{enumerate}
Following \cite{fujita,mits},
we use the following terminology.
The divisor $D$ is said to be \emph{rational}
if each $D_i$ is rational.
Let $0<T=\sum_{j=1}^t T_j\le D$ be a divisor,
where $T_j$'s are irreducible.
The divisor $T$ is called a \emph{twig} of $D$
if $(D-T_1)T_1=1$, $(D-T_j)T_j=2$ and $T_{j-1}T_j=1$ for $j\ge2$.
Suppose that $T$ is a rational twig of $D$.
The twig $T$ is said to be \emph{admissible} if $T_j^2<-1$ for all $j$.
We infer under the assumption (Z2) that
a rational twig is contractible if and only if it is admissible.
There exists an irreducible component $D_i$ of $D$
such that $T_{t}D_i=1$ and $D_i\nleq T$.
The rational twig $T$ is said to be \emph{maximal} if
$D_i$ is not rational or $(D-D_i)D_i>2$.
Suppose that the twig $T$ is rational and contractible.
The element
$\operatorname{Bk}(T)\in\mathbf{Q}(T)$ satisfying
$\operatorname{Bk}(T)T_j=(K+D)T_j$ for all $j$
is called the \emph{bark} of $T$.

Let $B$ be a connected component of $D$.
The divisor $B$ is called a \emph{rod}
if the dual graph $\Gamma(B)$ of $B$ is linear.
The divisor $B$ is called a \emph{rational fork}
if it satisfies the following conditions.
\begin{enumerate}
\item
$B=C+T^{(1)}+T^{(2)}+T^{(3)}$ is rational,
where
$T^{(1)},T^{(2)},T^{(3)}$ are contractible maximal rational twigs of $D$
and
$C$ is an irreducible curve such that $(B-C)C=3$.
\item
$(K+B-\Sigma_i\operatorname{Bk}(T^{(i)}))C<0$.
\end{enumerate}
Suppose that $B$ is a rational rod or a rational fork.
Suppose also that $B$ is contractible.
The element
$\operatorname{Bk}(B)\in\mathbf{Q}(B)$ satisfying
$\operatorname{Bk}(B)E=(K+B)E$ for all irreducible components $E$ of $B$
is called the \emph{bark} of $B$.
Let $B_{1},B_{2},\dots$ be the all
rational rods and rational forks which are contractible.
Let $T_{1},T_{2},\dots$ denote the all rational maximal twigs
which are contractible and not contained in any $B_{i}$.
The divisor $\operatorname{Bk}(D)=
\sum_{i}\operatorname{Bk}(B_{i})+
\sum_{j}\operatorname{Bk}(T_{j})$
is called the \emph{bark} of $D$.

Let $N$ denote the negative part of $K+D$.
We will use the following facts.
See \cite[Section 6]{fujita} and \cite[Chapter 1]{mits}.
Note that
rods (resp.\ rational forks) are called
\emph{clubs} (resp.\ \emph{abnormal rational clubs})
in \cite[Section 6]{fujita}.
The divisor $\operatorname{Bk}(D)$
is denoted by $\operatorname{Bk}^{\ast}(D)$
and called the \emph{thicker bark} of $D$.
\begin{lemma}\label{lem:zd2}
The following assertions hold
under the assumptions (Z1), (Z2).
\begin{enumerate}
\item
All rational rods, rational forks and rational twigs of $D$
belong to $\mathbf{Q}(N)$.
In particular, they are contractible.
\item
If $N\ne\operatorname{Bk}(D)$ and
any connected component of $D$ is not a rational rod,
then there exists a ($-1$)-curve $E\subset\operatorname{Supp}(N)$ such that
$DE\le1$ and $E\nleq D$.
Moreover, $E$ meets with $\operatorname{Supp}(\operatorname{Bk}(D))$ if $DE=1$.
\end{enumerate}
\end{lemma}
\begin{proof}
(i)
By \cite[Lemma 6.13]{fujita},
all rational rods and
rational twigs of $D$ belong to $\mathbf{Q}(N)$.
Let $B$ be a rational fork.
Write $B$ as $B=C+T^{(1)}+T^{(2)}+T^{(3)}$
as in the definition of rational forks.
Since $(N-\Sigma_i\operatorname{Bk}(T^{(i)}))C\le(K+D-\Sigma_i\operatorname{Bk}(T^{(i)}))C<0$,
we have $C\in\mathbf{Q}(N)$
by \cite[Lemma 6.15]{fujita}.
Hence $B\in\mathbf{Q}(N)$.

(ii)
We note that all rational twigs are admissible by (i) and (Z2).
Thus the assertion follows from \cite[Lemma 6.20]{fujita}.
See also Section 1.6, 1.7, 1.8 of \cite{mits}.
\end{proof}

\begin{lemma}\label{lem:2kd}
In addition to the assumptions (Z1), (Z2),
suppose that $N=\operatorname{Bk}(D)$, $\overline{\kappa}(V\setminus D)=2$
and that
every rational rod and rational fork of $D$
contains an irreducible component $E$ with $E^2<-2$.
Then the following assertions hold.
\begin{enumerate}
\item
$\lfloor N\rfloor=0$.
\item
$h^1(V,2K+D)=h^2(V,2K+D)=0$.
\item
$h^0(V,2K+D)=K(K+D)+D(K+D)/2+\chi(\mathcal{O}_V)$.
\end{enumerate}
\end{lemma}
\begin{proof}
We have
$\lfloor N\rfloor=\lfloor\operatorname{Bk}(D)\rfloor=0$
by \cite[Section 1.4 and Lemma 1.5]{mits}
and the assumption.
Since $\overline{\kappa}(V\setminus D)=2$,
we see $H^2>0$ by \cite{kawamata-class}.
We apply Theorem~\ref{thm:miyaoka} to $H$.
We have
$0=h^i(V,K+\lceil H\rceil)=h^{i}(V,2K+D+\lceil -N\rceil)=h^{i}(V,2K+D)$
for $i=1,2$.
The last equality follows from the Riemann-Roch formula.
\end{proof}
%%%%%%%%%%%%%%%%%%%%%%%%%%%%%%%%%%%%%%%%%%%%%%%%%%%%%%%%%%%%%%%%%%%%%%%%%%%%%%%
\subsection{The bigenus of $\mathbf{Q}$-homology planes}\label{ssec:homology}
%%%%%%%%%%%%%%%%%%%%%%%%%%%%%%%%%%%%%%%%%%%%%%%%%%%%%%%%%%%%%%%%%%%%%%%%%%%%%%%
Let $V$ be a smooth projective surface,
$K$ a canonical divisor
and $D\ne0$ an SNC-divisor on $V$.
Suppose that $X:=V\setminus D$ is
a $\mathbf{Q}$-homology plane.
That is, $H_i(X,\mathbf{Q})=\{0\}$ for $i>0$.
In this section, we compute the \emph{bigenus}
$h^0(V,2K+D)$ of $X$ (\cite{sakai:kdim}).
We will use the following facts
(\cite[Corollary 2.5, Theorem 2.8]{fujita},
\cite[Theorem II.4.2]{hartshorne:ample}
and \cite[Main Theorem]{mits:absence}).
\begin{lemma}\label{lem:fujita-h-mt}
The following assertions hold,
\begin{enumerate}
\item
$X$ is affine.
\item
$h^1(V,\mathcal{O}_V)=h^2(V,\mathcal{O}_V)=0$.
\item
$\Gamma(D)$ is a connected rational tree.
\item
$D$ is not contractible.
\item
If $\overline{\kappa}(X)=2$,
then $X$ does not contain topologically contractible algebraic curves.
\end{enumerate}
\end{lemma}

From now on,
we assume that $V$ and $D$ satisfy the following conditions.
\begin{itemize}
\item[(H1)]
$\overline{\kappa}(X)=2$.
\item[(H2)]
All ($-1$)-curves $E\le D$ satisfy $(D-E)E> 2$.
\end{itemize}

Note that $V,D$ satisfy the conditions (Z1), (Z2) in Section~\ref{ssec:zariski}.
Let $K+D=H+N$ be the Zariski decomposition,
where $N$ is the negative part of $K+D$.
\begin{proposition}\label{prop:gamma2}
Let $V$, $D$ be as above
satisfying the assumptions (H1), (H2).
Then the following assertions hold.
\begin{enumerate}
\item
$D$ is neither a rational rod nor a rational fork.
\item
We have $N=\operatorname{Bk}(D)$, $\lfloor N\rfloor=0$.
\item
We have $h^1(V,2K+D)=h^2(V,2K+D)=0$, $h^0(V,2K+D)=K(K+D)$.
\end{enumerate}
\end{proposition}
\begin{proof}
The divisor $D$ is neither a rational rod nor a rational fork
by Lemma \ref{lem:zd2} (i) and Lemma \ref{lem:fujita-h-mt} (iv).
Suppose $N\ne\operatorname{Bk}(D)$.
By the assertion (i) and Lemma~\ref{lem:zd2} (ii),
there exists a ($-1$)-curve $E\nleq D$
 such that $DE\le1$,
which contradicts
Lemma~\ref{lem:fujita-h-mt} (i), (v).
We have $\chi(\mathcal{O}_V)=1$, $D(K+D)=-2$
by Lemma~\ref{lem:fujita-h-mt} (ii), (iii).
Thus the remaining assertions follow from Lemma~\ref{lem:2kd}.
\end{proof}
%%%%%%%%%%%%%%%%%%%%%%%%%%%%%%%%%%%%%%%%%%%%%%%%%%%%%%%%%%%%%%%%%%%%%%%%%%%%%%%

We will use the following lemma
to show Theorem~\ref{thm2}.
\begin{lemma}\label{lem:phi}
Let $\varphi:V\rightarrow W$ be
the composition of successive blow-ups
over a smooth projective surface $W$
and $D\ne0$ an SNC-divisor on $V$.
Let $P\in W$ be the center of the first blow-up of $\varphi$
and $E$ the exceptional curve of the last blow-up over $P$.
Assume that the following conditions are satisfied.
\begin{itemize}
\item[(1)]
$V\setminus D$ is a $\mathbf{Q}$-homology plane with
$\overline{\kappa}(V\setminus D)=2$.
\item[(2)]
$E\nleq D$.
\end{itemize}
Then the following assertions hold.
\begin{enumerate}
\item
At least two locally irreducible branches of $\varphi(D)$
pass through $P$.
\item
Suppose that exactly two locally irreducible branches
$D_1,D_2$ of $\varphi(D)$ pass through $P$.
Suppose also that $D_1,D_2$ intersect each other at $P$ transversally.
Let $D_1',D_2'$ denote the proper transforms of $D_1,D_2$
via $\varphi$, respectively.
Then the dual graph of $D_1'+\varphi^{-1}(P)+D_2'$
has the following shape,
where $0\le T_1,T_2\le D$ may be empty.
\begin{center}
\includegraphics{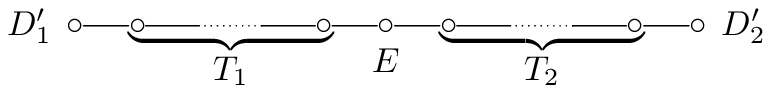}
\end{center}
\end{enumerate}
\end{lemma}
\begin{proof}
It is enough to show the assertions for the case in which
all blow-ups of $\varphi$ are done over $P$.

(i)
By Lemma~\ref{lem:fujita-h-mt} (iv), $\varphi(D)$ is not a point
(cf.\ \cite{mumford}).
Let $D'$ denote the proper transform of $\varphi(D)$ via $\varphi$.
The divisor $T:=\varphi^{-1}(P)$ is an SNC-divisor.
The dual graph of $T$ is a connected tree.
If $P\not\in\varphi(D)$,
then $T\subset V\setminus D$,
which contradicts the fact that $V\setminus D$ is affine.
Thus $P\in\varphi(D)$.
Suppose that there is only one locally irreducible branch of $\varphi(D)$
which passes through $P$.
Then $T\cap D'$ consists of a single point $Q$.

We see $T\ne E$ since $E\setminus D$ is not topologically contractible
by Lemma~\ref{lem:fujita-h-mt}.
Suppose that $T_0:=T-E$ does not intersect $D'$.
Then $D'$ intersects only $E\nleq D$ among the irreducible components of $T$.
Since $D$ is connected,
we have $T_0\subset V\setminus D$,
which is absurd.
Hence $T_0\cap D'=\{Q\}$.
We note that $E$ may pass through $Q$.
If $T_0E=1$, then
$E\cap D=\varnothing$ or $E\cap D=\{\text{one point}\}$,
which contradicts Lemma~\ref{lem:fujita-h-mt}.
Thus $T_0E\ge2$ and $T_0$ is not connected.
Let $T_1$ be a connected component of $T_0$
which does not passes through $Q$.
Then $T_1$ intersects only $E\nleq D$
among the irreducible components of $T-T_1$.
This means that $T_1\subset V\setminus D$,
which is impossible.

(ii)
It follows from (i) that
the dual graph of $D_1'+\varphi^{-1}(P)+D_2'$ is linear.
Suppose that
there exists an irreducible component $E'$ of $\varphi^{-1}(P)-E$
such that $E'\nleq D$.
If $E\cap E'\ne\varnothing$,
then $DE=1$, which contradicts Lemma~\ref{lem:fujita-h-mt}.
Thus $E\cap E'=\varnothing$.
Then $D_1'+\varphi^{-1}(P)+D_2'-E-E'$ is not connected.
Since $D$ is connected,
$V\setminus D$ contains a connected component of
$D_1'+\varphi^{-1}(P)+D_2'-E-E'$, which is impossible.
\end{proof}
%%%%%%%%%%%%%%%%%%%%%%%%%%%%%%%%%%%%%%%%%%%%%%%%%%%%%%%%%%%%%%%%%%%%%%%%%%%%%%%
\section{Proof of Theorem~\ref{thm1} and Theorem~\ref{thm2}}\label{sec:thm1}
%%%%%%%%%%%%%%%%%%%%%%%%%%%%%%%%%%%%%%%%%%%%%%%%%%%%%%%%%%%%%%%%%%%%%%%%%%%%%%%
Let $C$ be a rational cuspidal plane curve
and $P_1,\ldots,P_n$ the cusps of $C$.
We will use the fact that
$\mathbf{P}^2\setminus C$ is a $\mathbf{Q}$-homology plane.
Let $\sigma:V\rightarrow\mathbf{P}^2$ be the composition of a shortest sequence
of blow-ups such that the reduced total transform
$D:=\sigma^{-1}(C)$ is an SNC-divisor.
Let $C'$ be the proper transform of $C$.
For each $k$,
the dual graph of $\sigma^{-1}(P_k)+C'$ has the following shape.
\begin{center}
\includegraphics{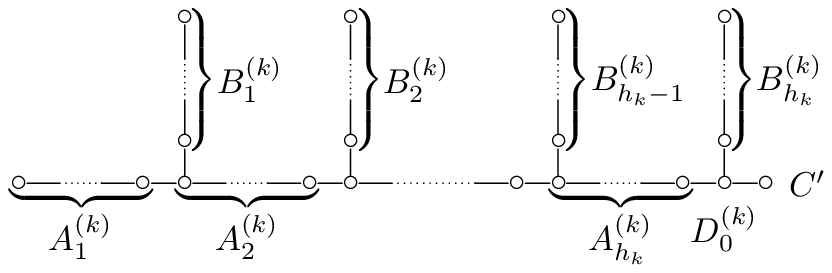}
\end{center}
Here $D_0^{(k)}$ is the exceptional curve of the last blow-up over $P_k$
and $h_k\ge 1$.
By definition, $A_1^{(k)}$ contains
the exceptional curve of the first blow-up over $P_k$.
The morphism $\sigma$ contracts $A_{h_k}^{(k)}+D_0^{(k)}+B_{h_k}^{(k)}$
 to a ($-1$)-curve $E$,
$A_{h_k-1}^{(k)}+E+B_{h_k-1}^{(k)}$ to a ($-1$)-curve,
and so on.
The self-intersection number of every irreducible component of
$A_i^{(k)}$ and $B_i^{(k)}$ is less than $-1$ for each $i$.
See \cite{bk,masa} for detail.

We give the graphs $A_1^{(k)},\ldots,A_{h_k}^{(k)}$
(resp.\ $B_1^{(k)},\ldots,B_{h_k}^{(k)}$)
the direction
from the left-hand side to the right
(resp.\ from the bottom to the top) in the above figure.
We assign each vertex
the self-intersection number of the
corresponding curve as its weight.
With these directions and weights,
we regard $A^{(k)}_i$ and $B^{(k)}_i$ as linear chains.
See Section~\ref{ssec:linear-chains}.

We may assume $\sigma=\sigma^{(n)}\circ\dots\circ\sigma^{(1)}$,
where $\sigma^{(k)}$ is the composition of the blow-ups over $P_k$ of $\sigma$.
There exists a decomposition
$\sigma^{(k)}=\sigma^{(k)}_0\circ\sigma^{(k)}_1\circ\cdots\circ\sigma^{(k)}_{h_k}$
such that $\sigma^{(k)}_i$ contracts $[A^{(k)}_i,1,B^{(k)}_i]$ to a ($-1$)-curve for each $i\ge 1$.
Let $\eta^{(k)}_i$ denote the number of the sprouting blow-ups
of $\sigma^{(k)}_i$ with respect to the ($-1$)-curve.
We will use the following lemma (\cite[Propositon 12]{tono:orevkov}).
\begin{lemma}\label{lem:cres}
We have
$A^{(k)}_i=\TW{\eta^{(k)}_i}\TA{B^{(k)\ast}_i}$,
${A^{(k)\ast}_i}=[B^{(k)}_i,\eta^{(k)}_i+1]$ for each $i,k$.
In particular, $A^{(k)}_i$ contains an irreducible component $E$
such that $E^2<-2$.
\end{lemma}
%
%
%
%
%%%%%%%%%%%%%%%%%%%%%%%%%%%%%%%%%%%%%%%%%%%%%%%%%%%%%%%%%%%%%%%%%%%%%%%%%%%%%%%
\subsection{Proof of Theorem~\ref{thm1}}
%%%%%%%%%%%%%%%%%%%%%%%%%%%%%%%%%%%%%%%%%%%%%%%%%%%%%%%%%%%%%%%%%%%%%%%%%%%%%%%
Let $K$ be a canonical divisor on $V$.
Let $\omega_k$ (resp.\ $\eta_k$)
denote the number of the subdivisional (resp.\ sprouting) blow-ups
of $\sigma$ over $P_k$, where the blow-up at $P_k$
is regarded as a subdivisional one.
We have $\eta_k=\sum_{i=1}^{h_k}\eta_i^{(k)}$ and
$\omega_k+\eta_k=\text{the number of the blow-ups of $\sigma^{(k)}$}$.
We complete the proof of Theorem~\ref{thm1}
by showing the following proposition.
We note that $\overline{\kappa}=2$ if $n\ge3$ by \cite{wakabayashi}.
\begin{proposition}\label{prop:n}
If $\overline{\kappa}=2$, then
\(
\displaystyle
0\le K(K+D)=7-2n-(C')^2-\sum_{k=1}^{n}\eta_k.
\)
Moreover,
we have $(C')^2\le 7-3n$.
\end{proposition}
\begin{proof}
We have $K(K+D)=7-D^2-\sum_{k=1}^{n}(\omega_k+\eta_k)$.
It was proved in \cite[Lemma~4]{masa} that
$D^2=(C')^2-\sum_{k=1}^{n}(\omega_k-2)$.
Thus we get $K(K+D)=7-2n-(C')^2-\sum_{k=1}^{n}\eta_k$.
The surface $V\setminus D$ is a $\mathbf{Q}$-homology plane
satisfying the conditions (H1), (H2) in Section~\ref{ssec:homology}.
By Proposition~\ref{prop:gamma2}, we have $0\le K(K+D)$.
The second blow-up of $\sigma$ over $P_k$
is a sprouting one for each $k$.
This fact shows the last inequality.
\end{proof}
%%%%%%%%%%%%%%%%%%%%%%%%%%%%%%%%%%%%%%%%%%%%%%%%%%%%%%%%%%%%%%%%%%%%%%%%%%%%%%%
\subsection{Proof of Theorem~\ref{thm2}}
%%%%%%%%%%%%%%%%%%%%%%%%%%%%%%%%%%%%%%%%%%%%%%%%%%%%%%%%%%%%%%%%%%%%%%%%%%%%%%%
Assume that $n=3$ and $(C')^2=-2$.
By \cite{wakabayashi}, we have $\overline{\kappa}=2$.
By Lemma~\ref{lem:cres} and Proposition~\ref{prop:n}, we get the following:
\begin{lemma}\label{lem:cr2}
The following assertions hold for each $k$.
\begin{enumerate}
\item
$h_k=1$ and $\eta^{(k)}_1=1$.
\item
$A^{(k)}_1=\TW{1}\TA{B^{(k)\ast}_1}$ and ${A^{(k)\ast}_1}=[B^{(k)}_1,2]$.
\end{enumerate}
\end{lemma}

Let $\sigma':V\rightarrow V'$ be the contraction of
$D^{(2)}_0$ and $D^{(3)}_0$.
Since $\sigma'(C')^2=0$,
there exists a $\mathbf{P}^1$-fibration $p':V'\rightarrow\mathbf{P}^1$
such that $\sigma'(C')$ is a nonsingular fiber.
Put $p=p'\circ\sigma':V\rightarrow\mathbf{P}^1$.
On $V$, there are two irreducible components of
$A^{(k)}_{1}+B^{(k)}_{1}$ meeting with $D^{(k)}_0$
for each $k$.
One of them must be a ($-2$)-curve and the other must not.
Let $D^{(k)}_1$ be the ($-2$)-curve and let $D^{(k)}_2$ be the remaining one.
Put $S_1=D^{(1)}_0$, $S_2=D^{(2)}_2$, $S_3=D^{(3)}_2$, $S_4=D^{(2)}_1$
and $S_5=D^{(3)}_1$.
The curves $S_1,\ldots,S_5$ are 1-sections of $p$.
Namely, the intersection number of them and a fiber is equal to one.
The divisor $D$ contains no other sections of $p$.
The divisor $F_0':=C'+D_0^{(2)}+D_0^{(3)}$
is a singular fiber of $p$.

The surface $X=V\setminus D$ is a $\mathbf{Q}$-homology plane.
A general fiber of $p|_{X}$ is isomorphic to a curve
$\mathbf{C}^{(4\ast)}=\mathbf{P}^1\setminus\{5\text{ points}\}$.
Cf.\ \cite{mits:noncomplete}.
There exists a birational morphism $\varphi:V\rightarrow\Sigma_d$
from $V$ onto the Hirzebruch surface $\Sigma_d$ for some $d\ge 0$.
The morphism $\varphi$ is the composition of the successive contractions
of the ($-1$)-curves in the singular fibers of $p$,
and $p\circ\varphi^{-1}:\Sigma_d\rightarrow\mathbf{P}^1$ is a
$\mathbf{P}^1$-bundle.

\begin{lemma}\label{lem:s1}
We may assume that
$\varphi(S_1+S_2+S_3)$ is smooth.
The following assertions hold.
\begin{enumerate}
\item
$\varphi(S_1)\sim \varphi(S_2)\sim \varphi(S_3)$ (linearly equivalent),
$d=\varphi(S_1)^2=0$.
\item
We have $\varphi(F_0')=\varphi(C')$.
The fiber $\varphi(F_0')$ passes through
$\varphi(S_2)\cap \varphi(S_4)$
and
$\varphi(S_3)\cap \varphi(S_5)$.
\item
$\varphi$ contains exactly one blow-up over $\varphi(S_1)$.
The set
$\varphi(S_1)\cap \varphi(S_4)\cap \varphi(S_5)$ consists of a single point,
which coincides with the center of the blow-up.
\item
$\varphi(S_4)^2=\varphi(S_5)^2=\varphi(S_4)\varphi(S_5)=2$,
$\varphi(S_4)\varphi(S_1)=\varphi(S_5)\varphi(S_1)=1$.
\end{enumerate}
\end{lemma}
\begin{proof}
By \cite[Lemma~17]{tono:bck2},
we may assume that
$\varphi(S_1+S_2+S_3)$ is smooth.
We have $\varphi(S_1)\sim\varphi(S_2)\sim\varphi(S_3)$
and $\varphi(S_1)^2=0$.
If $\varphi$ contracts $C'$,
then $\varphi(S_1+S_2+S_3)$ must be singular.
Thus $\varphi(F_0')=\varphi(C')$
and $\varphi$ contracts $D_0^{(2)}+D_0^{(3)}$.
Hence $\varphi(F_0')$ passes through
$\varphi(S_2)\cap \varphi(S_4)$ and $\varphi(S_3)\cap \varphi(S_5)$.

For $i=4,5$,
put $\varepsilon_i=\varphi(S_i)\varphi(S_1)$.
We have $\varphi(S_i)^2=2\varepsilon_i$
because $\varphi(S_i)\sim \varphi(S_1)+(\text{a fiber of }p\circ\varphi^{-1})\varepsilon_i$.
Since $\varphi(S_i)$ intersects $\varphi(S_{i-2})$,
we see $\varepsilon_i>0$.
Because $S_1^2=-1$,
$\varphi$ contains exactly one blow-up over $\varphi(S_1)$.
This means that $\varepsilon_4=\varepsilon_5=1$ and that
$\varphi(S_1)\cap\varphi(S_4)\cap\varphi(S_5)\ne\varnothing$.
We have $\varphi(S_4)\varphi(S_5)=\varepsilon_4+\varepsilon_5=2$.
The remaining assertions are clear.
\end{proof}

We use Lemma \ref{lem:s1} to show the following:

\begin{lemma}\label{lem:f1}
The following assertions hold.
\begin{enumerate}
\item
For $i=4,5$,
$\varphi$ contains exactly four blow-ups over $\varphi(S_i)$.
The centers of the blow-ups must be
the points of intersection of $\varphi(S_i)$
and the other sections $\varphi(S_j)$ ($j\ne i$).
\item
If a fiber $F$ of $p\circ\varphi^{-1}$
intersects $\varphi(S_1+\cdots+S_5)$ in five points,
then the proper transform $F'$ of $F$ via $\varphi$
is not a component of $D$ and
intersects $D^{(1)}_0$, $D^{(2)}_1$ and $D^{(3)}_1$.
\end{enumerate}
\end{lemma}
\begin{proof}
The first assertion of (i) follows from the fact that
$S_i^2=-2$ and $\varphi(S_i)^2=2$.
The second follows from Lemma~\ref{lem:s1}
and the fact that $S_i$ does not intersect the other sections
on $V$.
By (i) and Lemma~\ref{lem:s1} (iii),
$\varphi$ does not perform blow-ups over $F\cap\varphi(S_1+S_4+S_5)$.
This means that $F'$
intersects $D^{(1)}_0$, $D^{(2)}_1$ and $D^{(3)}_1$.
Since the dual graph of $D$ contains no loops,
$F'$ is not a component of $D$.
\end{proof}

By Lemma~\ref{lem:s1},
$\varphi(S_2)\cap\varphi(S_5)$ and $\varphi(S_3)\cap\varphi(S_4)$
consist of a single point, respectively.
\begin{lemma}\label{lem:f2}
The two points $\varphi(S_2)\cap\varphi(S_5)$,
$\varphi(S_3)\cap\varphi(S_4)$ are on a single fiber of $p\circ\varphi^{-1}$.
\end{lemma}
\begin{proof}
Suppose the contrary.
Let $F_2$ (resp.\ $F_3$)
denote the fiber of $p\circ\varphi^{-1}$ passing through
$\varphi(S_2)\cap\varphi(S_5)$
(resp.\ $\varphi(S_3)\cap\varphi(S_4)$).
Let $F_i'$ denote the proper transform of $F_i$ via $\varphi$ for $i=2,3$.
By Lemma~\ref{lem:s1} (iii),
$F_i'$ intersects $S_1$.
It follows from Lemma~\ref{lem:s1} and Lemma~\ref{lem:f1} (i) that
$F_i'$ intersects $D^{(i)}_1$.
Since the dual graph of $D$ does not contain loops,
$F_i'$ is not a component of $D$.

Suppose that $\varphi(S_4)\cap\varphi(S_5)$ consists of one point.
Let $F$ be the fiber of $p\circ\varphi^{-1}$ passing through
$\varphi(S_4)\cap\varphi(S_5)$.
By Lemma~\ref{lem:f1} (ii),
the proper transform via $\varphi$
of every fiber of $p\circ\varphi^{-1}$ other than
$F$ and $\varphi(F_0')$ is not a component of $D$.
This contradicts the fact that
there are three irreducible components
of $D-D^{(1)}_0$ meeting with $D^{(1)}_0$,
each of which is contained in a fiber of $p$.
Hence $\varphi(S_4)\cap\varphi(S_5)$ consists of two points.

Let $E$ be the exceptional curve of the blow-up of $\varphi$ at
$\varphi(S_1)\cap\varphi(S_4)\cap\varphi(S_5)$.
By Lemma~\ref{lem:s1} (iii) and Lemma~\ref{lem:f1} (i),
the proper transform $E'$ of $E$ by $\varphi$ intersects
$S_1$, $S_4$ and $S_5$.
Thus $E'$ is not a component of $D$.
By the same argument as above,
there are at most two components of $D-D^{(1)}_0$ meeting with $D^{(1)}_0$,
which is absurd.
\end{proof}

By Lemma~\ref{lem:s1} (iv),
the set $\varphi(S_4)\cap\varphi(S_5)$ consists of one or two points.
\begin{lemma}\label{lem:type1}
If $\varphi(S_4)\cap\varphi(S_5)$ consists of two points,
then $\deg C=4$.
\end{lemma}
\begin{proof}
We show $\sigma^{-1}(P_k)=D^{(k)}_0+D^{(k)}_1+D^{(k)}_2$ for each $k$.
Put
$\{Q_1\}=\varphi(S_1)\cap\varphi(S_5)$,
$\{Q_2\}=\varphi(S_2)\cap\varphi(S_5)$  and
$\{Q_3\}=\varphi(S_4)\cap\varphi(S_5)\setminus\{Q_1\}$.
For each $i$,
let $F_i$ be the fiber of $p\circ\varphi^{-1}$ passing through $Q_i$.
Put $\{Q_4\}=\varphi(S_3)\cap\varphi(S_4)$.
The fiber $F_2$ passes through $Q_4$ by Lemma~\ref{lem:f2}.
Since $S_1,\ldots,S_5$ do not intersect each other,
$\varphi$ performs blow-ups at all $Q_i$'s.
Let $E_i$ be the exceptional curve of the blow-up at $Q_i$.
We sometimes use the same symbols to denote the proper transforms
of $E_i$, $\varphi(S_i)$, etc.\ via blow-ups.
Let $F_i'$ denote the proper transform of $F_i$ via $\varphi$.

By Lemma~\ref{lem:s1} (iii) and Lemma~\ref{lem:f1} (i),
$\varphi$ does not perform blow-ups at
$E_1\cap\varphi(S_1)$, $E_1\cap\varphi(S_4)$ and $E_1\cap\varphi(S_5)$
after the blow-up at $Q_1$.
This means that $E_1$ intersects $S_1$, $S_4$ and $S_5$ on $V$.
Thus $E_1$ is not a component of $D$.
It follows from Lemma~\ref{lem:phi} (i) that
$\varphi$ does not perform blow-ups over $E_1$.
By Lemma~\ref{lem:f1} (ii),
the proper transforms via $\varphi$ of
the fibers of $p\circ\varphi^{-1}$
other than $\varphi(F_0')$, $F_1$, $F_2$ and $F_3$
are not contained in $D$.
By Lemma~\ref{lem:s1} (iii), $\varphi$ does not perform blow-ups over
$\varphi(S_1)\setminus\{Q_1\}$.
Thus $F_2'$ and $F_3'$ are contained in $D$ and intersect $S_1=D_0^{(1)}$.
It follows that $F_2+F_3$ coincides with
the image of $D^{(1)}_1+D^{(1)}_2$ under $\varphi$.
We have $(F_3')^2<-2$
since $\varphi$ performs blow-ups
at $Q_3$, $\varphi(S_2)\cap F_3$ and $\varphi(S_3)\cap F_3$.
Hence $F_2'=D^{(1)}_1$ and $F_3'=D^{(1)}_2$.

Since $(F_2')^2=-2$,
$\varphi$ only performs two blow-ups over $F_2$.
The centers coincide with $Q_2$ and $Q_4$.
The curve $E_4$ is not a component of $D$
because it intersects $S_4=D^{(2)}_1$ and $F_2'=D^{(1)}_1$ on $V$.
By Lemma~\ref{lem:phi} (i),
$\varphi$ does not perform blow-ups over $E_4$.
Similarly, $E_2$ is not a component of $D$
and $\varphi$ does not perform blow-ups over $E_2$.
It follows that $D^{(1)}_1$ intersects only
$D^{(1)}_0$ among the components of $D-D^{(1)}_1$.
Hence $(F_3')^2=(D^{(1)}_2)^2=-3$.

We see that
$\varphi$ does not perform blow-ups over
$F_3\setminus(\varphi(S_2+S_3)\cup\{Q_3\})$.
By Lemma~\ref{lem:f1} (i),
$\varphi$ does not perform blow-ups at
$E_3\cap\varphi(S_4)$ and $E_3\cap\varphi(S_5)$
after the blow-up at $Q_3$.
Thus $E_3$ intersects $F_3'$, $S_4$ and $S_5$ on $V$.
Hence $E_3$ is not a component of $D$.
It follows that $S_{i+2}=D^{(i)}_1$ intersects only
$D^{(i)}_0$ among the components of $D-D^{(i)}_1$ for $i=2,3$.
Hence $S_i^2=(D^{(i)}_2)^2=-3$.

For $i=2,3$,
let $E_{3,i}$ be the exceptional curve of the blow-up at $F_3\cap\varphi(S_i)$.
Since $S_i^2=(F_3')^2=-3$,
$\varphi$ does not perform blow-ups at
$E_{3,i}\cap\varphi(S_i)$ and $E_{3,i}\cap F_3$
after the blow-up at $F_3\cap\varphi(S_i)$.
This means that $E_{3,i}$ intersects $F_3'=D^{(1)}_2$ and $S_i=D^{(i)}_2$
on $V$.
Thus $E_{3,i}$ is not a component of $D$.
Hence $D^{(1)}_2$ intersects only
$D^{(1)}_0$ among the components of $D-D^{(1)}_2$.
Since $S_i^2=-3$,
$\varphi$ does not perform blow-ups at
$F_1\cap\varphi(S_2)$ and $F_1\cap\varphi(S_3)$.
Thus $F_1'$ is not a component of $D$.
Hence $D^{(i)}_2$ intersects only
$D^{(i)}_0$ among the components of $D-D^{(i)}_2$.
\end{proof}

From now on,
we assume that
$\varphi(S_4)\cap\varphi(S_5)$ consists of one point.
We prove that the assumption causes a contradiction.
Let $F_1$ (resp.\ $F_2$) be the fiber of $p\circ\varphi^{-1}$
passing through $\varphi(S_4)\cap\varphi(S_5)$
(resp.\ $\varphi(S_4)\cap\varphi(S_3)$).
For each $i$, put $T_i=\varphi^{-1}(F_i)$.
Let $b_i$ be the number of the irreducible components
of $T_i$ which are not contained in $D$.
\begin{lemma}\label{lem:b}
The fibration $p$ has exactly three singular fibers $F_0',T_1,T_2$.
We have $b_1+b_2=6$.
\end{lemma}
\begin{proof}
It follows from Lemma~\ref{lem:f1},
Lemma~\ref{lem:f2} and Lemma~\ref{lem:phi} that
$p$ has exactly three singular fibers $F_0',T_1,T_2$.
Let $\rho(V)$ denote the Picard number of $V$
and $r(D)$ the number of the irreducible components of $D$.
We have $\rho(V)=r(D)$.
The number of the blow-ups of $\varphi$ is equal to
$r(D)+b_1+b_2-\text{(the number of the sections in $D$)}-\text{(the number of the singular fibers)}$.
Thus $r(D)=\rho(V)=\rho(\Sigma_d)+r(D)+b_1+b_2-8$.
\end{proof}

\begin{figure}[t]
\begin{center}
\includegraphics{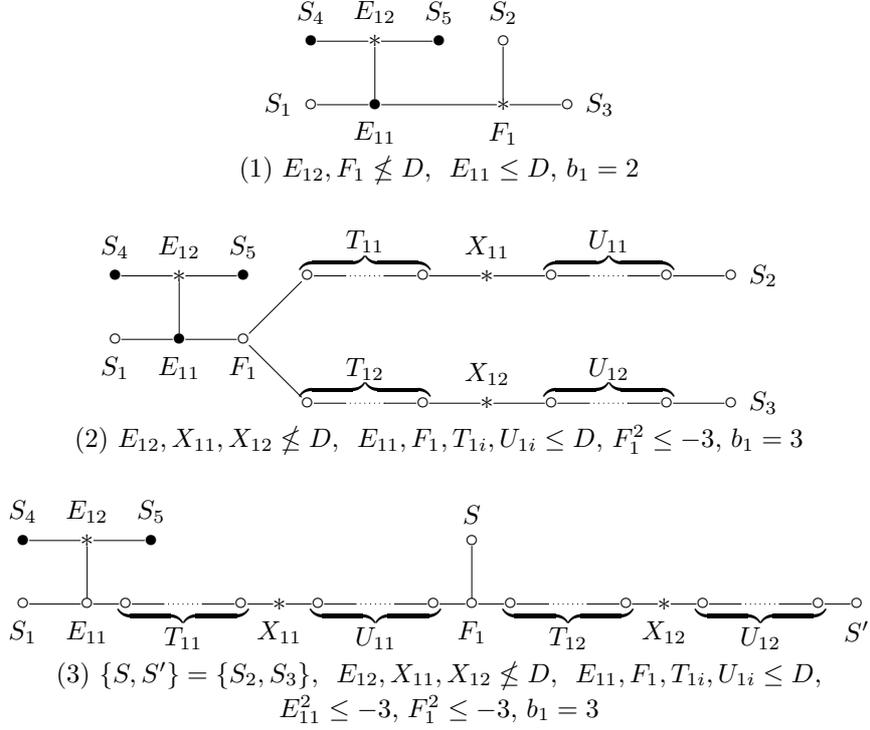}
\caption{The weighted dual graph of $T_1+S_1+\cdots+S_5$}
\label{fig:t1}
\end{center}
\end{figure}
\begin{figure}
\begin{center}
\includegraphics{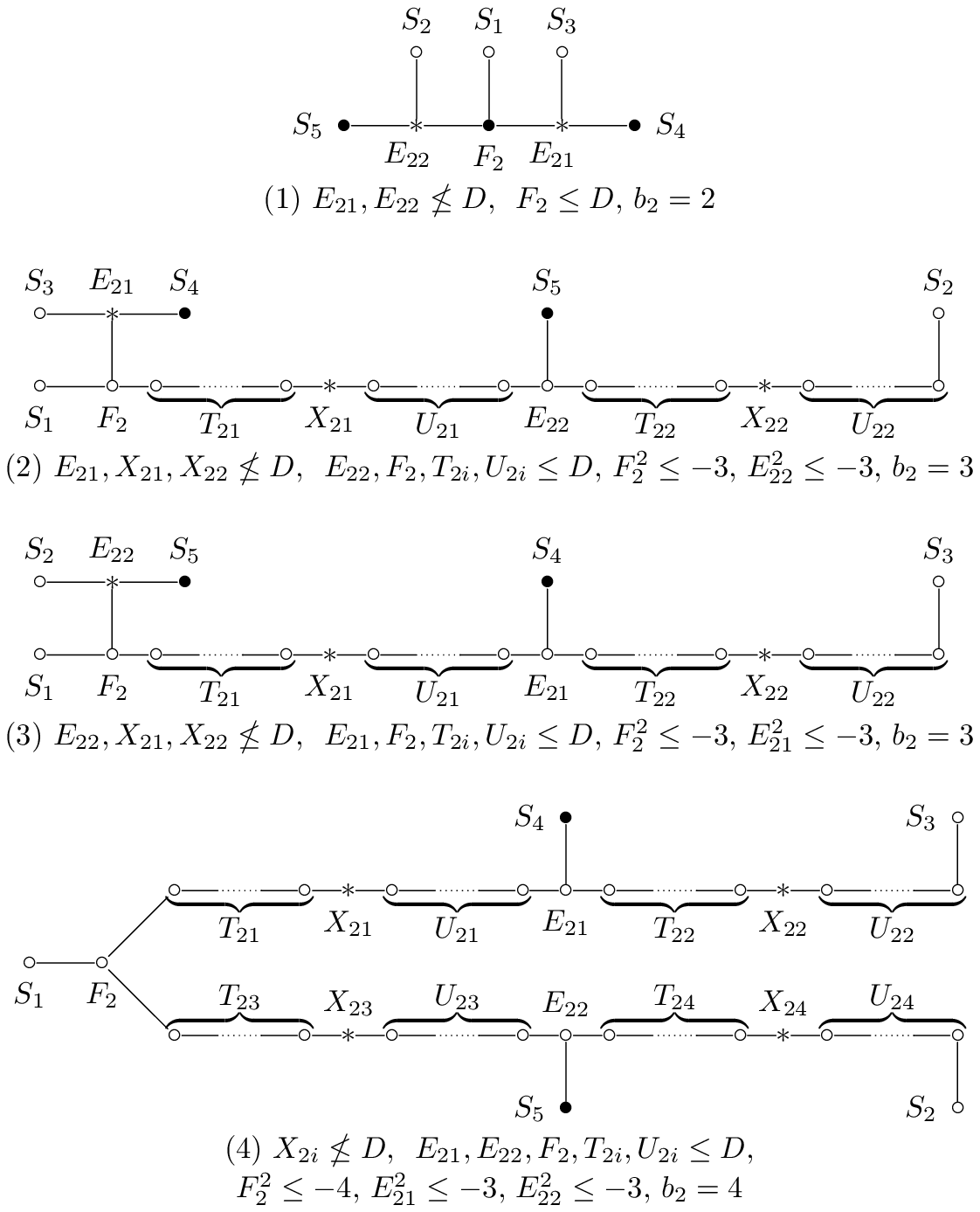}
\caption{The weighted dual graph of $T_2+S_1+\cdots+S_5$}
\label{fig:t2}
\end{center}
\end{figure}
\begin{figure}
\begin{center}
\includegraphics{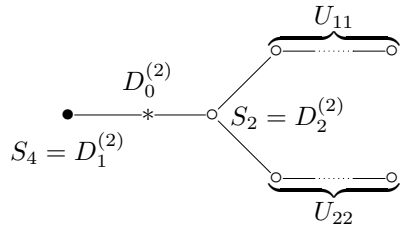}
\caption{The weighted dual graph of $\sigma^{-1}(P_2)$}
\label{fig:d}
\end{center}
\end{figure}

The next lemma describes the structure of the fibration $p$.

\begin{lemma}\label{lem:t}
We may assume that
the weighted dual graph of
$T_1+S_1+\cdots+S_5$
(resp.\ $T_2+S_1+\cdots+S_5$)
coincides with that in Figure~\ref{fig:t1} (2)
(resp.\ Figure~\ref{fig:t2} (2)).
The weighted dual graph of $\sigma^{-1}(P_2)$
coincides with that in Figure~\ref{fig:d}.
In the figures,
$\ast$ (resp.\ $\bullet$) denotes a ($-1$)-curve
(resp.\ ($-2$)-curve).
\end{lemma}
\begin{proof}
We first show that the weighted dual graph of
$T_1+S_1+\cdots+S_5$
coincides with one of those in Figure~\ref{fig:t1}.
Let $E_{ij}$ be the exceptional curve
of the $j$-th blow-up $\varphi_{ij}$ of $\varphi$ over $F_i$
for $i=1,2$.
We use the same symbols to denote the proper transforms
of $E_{ij}$, $\varphi(S_j)$, etc.\ via blow-ups.
Since $S_4=D_1^{(2)}$ and $S_5=D_1^{(3)}$ do not intersect each other,
we may assume that the centers of $\varphi_{11}$ and $\varphi_{12}$ are
$\varphi(S_4)\cap\varphi(S_5)$.

By Lemma~\ref{lem:f1} (i),
$\varphi$ does not perform blow-ups at
$E_{12}\cap\varphi(S_4)$ and $E_{12}\cap\varphi(S_5)$.
This means that $E_{12}$
intersects $S_4=D_1^{(2)}$ and $S_5=D_1^{(3)}$ on $V$.
Thus $E_{12}\nleq D$.
By Lemma~\ref{lem:phi} (i),
$\varphi$ does not perform blow-ups over $E_{12}$.
By Lemma~\ref{lem:s1} (iii),
$E_{11}$ and $F_2$ intersect $S_1$ on $V$.
Let $E$ be an irreducible components of a fiber of $p$
which meets with $D_0^{(1)}$.
It follows from Lemma~\ref{lem:f1} (ii) that $E$ is not a component of $D$
if $E\ne C', E_{11}, F_2$.
Since three irreducible components of $D-D_0^{(1)}$ meet with $D_0^{(1)}$,
the curves $E_{11}$, $F_2$ must be components of $D$.
By Lemma~\ref{lem:phi} (i),
$\varphi$ does not perform blow-ups over $E_{11}\setminus F_1$.

Suppose $F_1\nleq D$.
By Lemma~\ref{lem:phi} (i),
$\varphi$ does not perform further blow-ups over $F_1$.
We have $E_{11}=D^{(1)}_1$.
The weighted dual graph of $T_1+S_1+\dots+S_5$ coincides with (1) in this case.
Suppose $F_1\le D$.
Since $D$ does not contain loops,
$\varphi$ performs blow-ups at two or three points of
the three points $F_1\cap(E_{11}+\varphi(S_2+S_3))$.
If the latter case occurs,
then it follows from Lemma~\ref{lem:phi} (ii) that
$D$ is not connected, which is absurd.
Thus the former case must occur.
It follows from Lemma~\ref{lem:phi} (ii) that
the weighted dual graph of $T_1+S_1+\dots+S_5$
coincides with (2) (resp.\ (3))
if $\varphi$ does not perform (resp.\ performs) a blow-up at $F_1\cap E_{11}$.

We next show that the weighted dual graph of
$T_2+S_1+\cdots+S_5$
coincides with one of those in Figure~\ref{fig:t2}.
By Lemma~\ref{lem:f1} (i),
we may assume that the center of $\varphi_{21}$ (resp.\ $\varphi_{22}$)
is  $\varphi(S_4)\cap F_2$ (resp.\ $\varphi(S_5)\cap F_2$).
Furthermore,
$\varphi$ does not perform blow-ups at
$\varphi(S_4)\cap E_{21}$ and $\varphi(S_5)\cap E_{22}$.
If $\varphi$ does not perform blow-ups over $F_2$ further,
then
the weighted dual graph of $T_2+S_1+\dots+S_5$ coincides with (1).
We have $E_{21},E_{22}\nleq D$ in this case.

Suppose that $\varphi$ performs blow-ups over $F_2$ further.
If $\varphi$ does not perform blow-ups over $E_{21}$,
then $E_{21}$ intersects $S_3=D_2^{(3)}$ and $S_4=D_1^{(2)}$ on $V$.
Thus $E_{21}\nleq D$.
If $\varphi$ performs a blow-up over $E_{21}$,
then it must be done over
$E_{21}\cap\varphi(S_3)$ or $E_{21}\cap F_2$ by Lemma~\ref{lem:phi} (i).
Moreover, we have $E_{21}\le D$.
Since $D$ contains no loops,
it must be done over both of
$E_{21}\cap\varphi(S_3)$ and $E_{21}\cap F_2$.
Similar arguments are valid for $E_{22}$.
If $\varphi$ performs blow-ups over
both of $E_{21}$ and $E_{22}$,
then
the weighted dual graph of $T_2+S_1+\dots+S_5$ coincides with (4).
Otherwise it coincides with (2) or (3).

Now we show that the weighted dual graph of $\sigma^{-1}(P_2)$ coincides with
that in Figure~\ref{fig:d}.
We say that $p$ is of type ($i$-$j$)
if $T_1$ (resp.\ $T_2$) coincides with
($i$) in Figure~\ref{fig:t1}
(resp.\ ($j$) in Figure~\ref{fig:t2}).
We prove that $p$ is of type (2--2) or (2--3).
By Lemma~\ref{lem:b},
$p$ must be of type (1--4), (2--2), (2--3), (3--2) or (3--3).
Suppose $p$ is of type (1--4).
We have $A^{(1)}_{1}=[F_2,\ldots]$ and $B^{(1)}_{1}=[2]$.
By Lemma~\ref{lem:cr2},
$A^{(1)}_{1}=[B^{(1)}_{1},2]^{\ast}=[3]$.
This contradicts the fact that $F_2^2\le-4$ on $V$.
Suppose $p$ is of type (3--2) or (3--3).
In this case, we have $\{D^{(1)}_1,D^{(1)}_2\}=\{E_{11},F_2\}$.
But $E_{11}^2\le-3$ and $F_2^2\le-3$,
which contradicts $(D^{(1)}_1)^2=-2$.
By changing the roles of $P_2$ and $P_3$,
if necessary, we may assume that $p$ is of type (2--2).
It follows that
the weighted dual graph of $\sigma^{-1}(P_2)$ coincides with
that in Figure~\ref{fig:d}.
\end{proof}
We can arrange the order of the blow-ups of $\varphi$ such that
$\varphi=\varphi_0\circ\varphi_{11}\circ\varphi_{12}\circ\varphi_{22}\circ\varphi_{21}$.
Here $\varphi_{ij}$
contracts $T_{ij}+X_{ij}+U_{ij}$ to a point
in Figure~\ref{fig:t1} (2) and Figure~\ref{fig:t2} (2).
The morphism $\varphi_0$
contracts $E_{11}+E_{12}+E_{21}+E_{22}+D_0^{(2)}+D_0^{(3)}$
to points.
We use the same symbols to denote the proper transforms
of $\varphi(C')$, $\varphi(S_j)$, etc.\ via blow-ups.
The morphism $\varphi_0$ performs the blow-ups at
$\varphi(C')\cap\varphi(S_2)$
and
$F_2\cap\varphi(S_2)$.
The morphism $\varphi_{11}$ performs the blow-up at
$F_1\cap\varphi(S_2)$
and
$\varphi_{22}$ performs the blow-up at
$E_{22}\cap\varphi(S_2)$.
Thus $S_2^2\le -4$.
By Lemma~\ref{lem:t},
we have $A^{(2)}_1=[\ldots,-S_2^2]$ and $B^{(2)}_1=[2]$.
By Lemma~\ref{lem:cr2}, we see $A^{(2)}_1=\TW{1}\TA B^{(2)\ast}_1=[3]$.
Hence $S_2^2=-3$, which is a contradiction.
Thus $\varphi(S_4)\cap\varphi(S_5)$ must consist of two points.
We have completed the proof of Theorem~\ref{thm2}.
%%%%%%%%%%%%%%%%%%%%%%%%%%%%%%%%%%%%%%%%%%%%%%%%%%%%%%%%%%%%%%%

%
%
%
%
\begin{flushleft}
\small
\textit{E-mail address}:  \texttt{ktono@rimath.saitama-u.ac.jp}
\end{flushleft}
\begin{flushleft}
\scshape
\small
Department of Mathematics,
Graduate School of Science and Engineering,
Saitama University,
Saitama-City, Saitama 338--8570,
Japan.
\end{flushleft}
\end{document}